\documentclass{article}
\usepackage[final]{graphicx}
\usepackage{psfrag}
\usepackage{amsmath,amsfonts,amssymb,amsxtra,subeqnarray}

\newcommand {\Real}{\ensuremath{{\mathbb{R}}}}
\newcommand {\Natural}{\ensuremath{{\mathbb{N}}}}
\newcommand {\Complex}{\ensuremath{{\mathbb{C}}}}

\newcommand{\R}{\ensuremath{\mathcal R}}

\newcommand{\setS}{\ensuremath{\mathcal S}}

\newcommand{\U}{\ensuremath{\mathcal U}}

\newcommand{\X}{\ensuremath{\mathcal X}}

\newcommand{\N}{\ensuremath{\mathcal N}}

\newcommand{\nal}{{\mbox{\rm null}}}

\newcommand{\xhat}{\hat x}

\newtheorem{theorem}{Theorem}
\newtheorem{algorithm}{Algorithm}
\newtheorem{corollary}{Corollary}

\newtheorem{lemma}{Lemma}

\newtheorem{definition}{Definition}
\newtheorem{remark}{Remark}
\newtheorem{assumption}{Assumption}

\newtheorem{proposition}{Proposition}
\newenvironment{proof}{\noindent {\bf Proof.}}{\hfill \hspace*{1pt}\hfill$\blacksquare$}

\begin{document}

\title{Deadbeat control: construction via sets}
\author{S. Emre Tuna\footnote{Author is with Department of
Electrical and Electronics Engineering, Middle East Technical
University, 06800 Ankara, Turkey. Email: {\tt
tuna@eee.metu.edu.tr}}} \maketitle

\begin{abstract}
A geometric generalization of the discrete-time linear deadbeat
control problem is studied. The proposed method to generate a
deadbeat tracker for a given nonlinear system is constructive and
makes use of sets that can be computed iteratively. For
demonstration, derivations of deadbeat feedback law and tracker
dynamics are provided for various example systems. Based on the
method, a simple algorithm that computes the deadbeat gain for a
linear system with scalar input is given.
\end{abstract}

\section{Introduction}

Deadbeat control (regulation) problem for the discrete-time linear
system
\begin{eqnarray}\label{eqn:standard}
\xhat^{+}=A\xhat+Bu
\end{eqnarray}
concerns with finding a (linear) feedback law $u=-K\xhat$ such
that any solution of the closed-loop system hits the origin in
finite time. Thanks to linearity, the same feedback gain $K$ can
be used to make the solution $\xhat(\cdot)$ of
system~\eqref{eqn:standard} (exactly) converge to the solution
$x(\cdot)$ of the autonomous system $x^{+}=Ax$ in finite time.
Note that this time $u=K(x-\xhat)$. The problem being solved in
this case is deadbeat tracking. Though deadbeat regulation and
deadbeat tracking are equivalent problems for linear systems, the
latter subsumes the former when the systems are nonlinear. In this
paper we interest ourselves with the nonlinear deadbeat tracking
problem. Namely, given two systems, one of them autonomous and the
other with a control input, we attempt to find a method to
generate a feedback law to couple two systems so that their
solutions become equal after some finite time. Our point of
departure for generalization however is not \eqref{eqn:standard}
but a slight modification of it. Namely,
\begin{eqnarray}\label{eqn:sensible}
\xhat^{+}=A(\xhat+Bu)
\end{eqnarray}
which is the form\footnote{Forms \eqref{eqn:standard} and
\eqref{eqn:sensible} are equivalent from the deadbeat control
point of view. See Remark~\ref{rem:equivalent}.} we adopt for
generalization. The main reason is that the tools we use in our
analysis suggest \eqref{eqn:sensible} as a more natural choice.

Deadbeat control theory for linear systems is acknowledged to have
been well-established \cite{oreilly81,kucera84}. Different
formulations provided different techniques to compute the deadbeat
feedback gain \cite{franklin82,lewis82,sugimoto93}. As for
discrete-time nonlinear systems, the problem seems to have
attracted fewer researchers. Among the cases being studied are
bilinear systems \cite{grasselli80}, polynomial systems
\cite{nesic98a,nesic98b}, and, as a subclass of the latter,
Wiener-Hammerstein systems \cite{nesic99}.

The toy example that we keep in the back of our mind while we
attempt to reach a generalization is the simple case where $A$ is
a rotation matrix in $\Real^{2}$
\begin{eqnarray*}
A =\left[\!\!\begin{array}{rr} \cos\theta &\sin\theta\\
-\sin\theta &\cos\theta
\end{array}\!\!\right]
\end{eqnarray*}
with angle of rotation $\theta$ different from $0$ and $\pi$.
Letting $B=[1\ \ 0]^{T}$, the deadbeat tracker turns out to be
\begin{eqnarray*}
\xhat^{+}=A(\xhat+B[1\ \ -\cot\theta](x-\xhat))
\end{eqnarray*}
Now we state the key observation in this paper: The term in
brackets is the intersection of two equivalence classes (sometimes
called congruence classes \cite{lax96}). Namely,
\begin{eqnarray*}
\xhat+B[1\ \ -\cot\theta](x-\xhat)=(\xhat+{\rm
range}(B))\cap(x+A^{-1}{\rm range}(B))
\end{eqnarray*}
as shown in Fig.~\ref{fig:intersect}.
\begin{figure}[h]
\begin{center}
\includegraphics[scale=0.6]{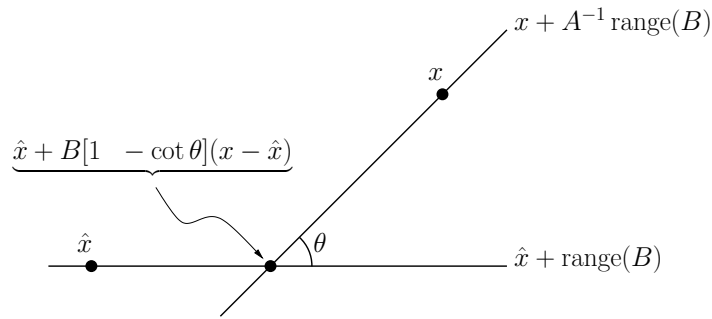}
\caption{Intersection of two equivalence
classes.}\label{fig:intersect}
\end{center}
\end{figure}
Based on this observation, the contribution of this paper is
intended to be in showing that such equivalence classes can be
defined even for nonlinear systems of arbitrary order, which in
turn allows one to construct deadbeat feedback laws and hence
deadbeat trackers provided that certain conditions
(Assumption~\ref{assume:X} and Assumption~\ref{assume:invariance})
hold. We now note and later demonstrate that when the system is
linear those assumptions are minimal for a deadbeat observer to
exist. We note that the approach in this paper is the {\em dual}
of the approach adopted in \cite{tuna11}.

The remainder of the paper is organized as follows. Next section
contains some preliminary material. In Section~\ref{sec:def} we
give the formal problem definition. Section~\ref{sec:sets} is
where we describe the sets that we use in construction of the
deadbeat tracker. We state and prove the main result in
Section~\ref{sec:main}. We provide examples in
Section~\ref{sec:ex}, where we construct deadbeat observers for
two different third order systems. In Section~\ref{sec:alg} we
present an algorithm to compute the deadbeat control gain for a
linear system with scalar input.

\section{Preliminaries}

Identity matrix is denoted by $I$. Null space and range space of a
matrix $M\in\Real^{m\times n}$ are denoted by $\N(M)$ and $\R(M)$,
respectively. Given map $f:\X\to\X$, $f^{-1}(\cdot)$ denotes the
{\em inverse} map in the general sense that for $\setS\subset\X$,
$f^{-1}(\setS)$ is the set of all $x\in\X$ satisfying
$f(x)\in\setS$. That is, we will not need $f$ be bijective when
talking about its inverse. Linear map $x\mapsto Ax$ will not be
exempt from this notation. Therefore, unless otherwise stated, the
reader should not assume that $A$ is a nonsingular matrix when we
write $A^{-1}$. The set of nonnegative integers is denoted by
$\Natural$ and $\Real_{>0}$ denotes the set of strictly positive
real numbers.

\section{Problem definition}\label{sec:def}

Given maps $f:\X\to\X$ and  $\mu:\X\times\U\to\X$, consider the
following discrete-time system
\begin{eqnarray}\label{eqn:system}
\xhat^{+}&=&f(\mu(\xhat,\,u))
\end{eqnarray}
where $\xhat\in\X\subset\Real^{n}$ is the {\em state} and
$u\in\U\subset\Real^{m}$ is the {\em (control)  input}. Map $\mu$
is assumed to satisfy
\begin{eqnarray}\label{eqn:obvious}
x\in\mu(x,\,\U)\qquad\forall x\in\X\,.
\end{eqnarray}
Notation $\xhat^{+}$ denotes the state at the next time instant.
The goal is to make system~\eqref{eqn:system} (by choosing proper
input values) follow the autonomous system $x^{+}=f(x)$ in
deadbeat fashion. We suppose that we have access to the full state
information $x$ of the autonomous system. Then the problem becomes
to construct some feedback law
$\kappa:\X\times\X\rightrightarrows\U$ such that the states of the
below coupled systems
\begin{subeqnarray}\label{eqn:cascade}
x^{+}&=&f(x)\\
\xhat^{+}&\in&f(\mu(\xhat,\,\kappa(\xhat,\,x)))
\end{subeqnarray}
converge to each other in finite time. The {\em solution} of
system~(\ref{eqn:cascade}a) at time $k\in\Natural$, having started
at initial condition $x\in\X$, is denoted by $\phi(k,\,x)$. Note
that $\phi(0,\,x)=x$ and $\phi(k+1,\,x)=f(\phi(k,\,x))$ for all
$x$ and $k$. {\em A} solution of system~(\ref{eqn:cascade}b) is
denoted by $\psi(k,\,\xhat,\,x)$. For the formal problem
description we need the definition below.

\begin{definition}
Map $\kappa:\X\times\X\rightrightarrows\U$ is said to be a {\em
deadbeat feedback law for system~\eqref{eqn:system}} if there
exists $p\geq 1$ such that all solutions of coupled
systems~\eqref{eqn:cascade} satisfy
\begin{eqnarray*}
\psi(k,\,\xhat,\,x)=\phi(k,\,x)
\end{eqnarray*}
for all $x,\,\xhat\in\X$ and $k\geq p$. System~{\em
(\ref{eqn:cascade}b)} then is said to be a {\em deadbeat tracker}.
\end{definition}

\begin{definition}
System~\eqref{eqn:system} is said to be {\em deadbeat
controllable} if there exists a deadbeat feedback law for it.
\end{definition}

In this paper we present a procedure to construct a deadbeat
tracker from system~\eqref{eqn:system} provided that certain
conditions (Assumption~\ref{assume:X} and
Assumption~\ref{assume:invariance}) hold. Our construction will
make use of some sets, which we define in the next section. Before
moving on into the next section, however, we choose to remind the
reader of a standard fact regarding the controllability of linear
systems. Then we provide a Lemma~\ref{lem:subspace} as a geometric
equivalent of that well-known result. Lemma~\ref{lem:subspace}
will find use later when we attempt to interpret and display the
generality of the assumptions we will have made.

The following criterion, known as Popov-Belevitch-Hautus (PBH)
test, is an elegant tool for checking (deadbeat) controllability.

\begin{proposition}[PBH test]
System~\eqref{eqn:standard} with $A\in\Real^{n\times n}$ and
$B\in\Real^{n\times m}$ is deadbeat controllable if and only if
\begin{eqnarray}\label{eqn:PBH}
{\rm rank}\left[A-\lambda I\ \ B\right]=n\quad \mbox{for all}\quad
\lambda\neq 0
\end{eqnarray}
where $\lambda$ is a complex scalar.
\end{proposition}

\begin{remark}\label{rem:equivalent}
From PBH test it readily follows that system~\eqref{eqn:sensible}
is deadbeat controllable if and only if
system~\eqref{eqn:standard} is deadbeat controllable. In
particular, if $K\in\Real^{m\times n}$ is a deadbeat feedback gain
for system~\eqref{eqn:sensible} then $KA$ is a deadbeat feedback
gain for system~\eqref{eqn:standard}.
\end{remark}

The below result is a geometric equivalent of PBH test
\cite{eldem94,mullis72}.

\begin{lemma}\label{lem:subspace}
Given $A\in\Real^{n\times n}$ and $B\in\Real^{n\times m}$, let
subspace $\setS_{-k}$ of $\Real^{n}$ be defined as
$\setS_{-k-1}:=A^{-1}\setS_{-k}+\setS_{0}$ for $k=0,\,1,\,\ldots$
with $\setS_{0}:=\R(B)$. Then system~\eqref{eqn:sensible} is
deadbeat controllable if and only if
\begin{eqnarray}\label{eqn:SET}
\setS_{-n}=\Real^{n}\,.
\end{eqnarray}
\end{lemma}

\begin{proof}
For simplicity we provide the demonstration for the case where
each $\setS_{-k}$ is a subspace of $\Complex^{n}$ (over field
$\Complex$). The case $\setS_{-k}\subset\Real^{n}$ is a little
longer to prove yet it is true.

We first show \eqref{eqn:SET}$\implies$\eqref{eqn:PBH}. Suppose
\eqref{eqn:PBH} fails. That is, there exists a (left) eigenvector
$w\in\Complex^{n}$ and a nonzero eigenvalue $\lambda\in\Complex$
such that $w^{T}A=\lambda w^{T}$ and $w^{T}B=0$. Now suppose for
some $k$ we have $w\perp\setS_{-k}$. That is, $w^{T}v=0$ for all
$v\in\setS_{-k}$. We claim that $w\perp A^{-1}\setS_{-k}$. Suppose
not. Then one can find $v\in A^{-1}\setS_{-k}$ such that
\begin{eqnarray}\label{eqn:tobecontradicted}
w^{T}v\neq 0\,.
\end{eqnarray}
Also, since $Av\in\setS_{-k}$ we can write
\begin{eqnarray*}
0
&=&w^{T}Av\\
&=&\lambda w^{T}v
\end{eqnarray*}
which contradicts \eqref{eqn:tobecontradicted} since $\lambda\neq
0$. Hence our claim holds. Moreover, since $w^{T}B=0$, we can
write $w\perp\setS_{0}$. Consequently, $w\perp
A^{-1}\setS_{-k}+\setS_{0}=\setS_{-k-1}$. We have established
therefore
\begin{eqnarray}\label{eqn:toimply}
w\perp\setS_{-k}\implies w\perp\setS_{-k-1}\,.
\end{eqnarray}
Recall that $w\perp\setS_{0}$. That means by \eqref{eqn:toimply}
that $w\perp\setS_{-k}$ for all $k$. Hence \eqref{eqn:SET} fails.

Now we demonstrate the other direction
\eqref{eqn:PBH}$\implies$\eqref{eqn:SET}. Note first that for any
subspace $\setS$ we can write
$(A^{-1}\setS)^{\perp}=A^{T}\setS^{\perp}$. Therefore equation
$\setS_{-k-1}=A^{-1}\setS_{-k}+\setS_{0}$ yields
\begin{eqnarray}\label{eqn:dual}
\setS_{-k-1}^{\perp}=A^{T}\setS_{-k}^{\perp}\cap\setS_{0}^{\perp}\,.
\end{eqnarray}
Then, since
\begin{eqnarray*}
\setS_{-k}=\R(B)+A^{-1}\R(B)+A^{-2}\R(B)+\ldots+A^{-k}\R(B)
\end{eqnarray*}
we have $\setS_{-k-1}\supset\setS_{-k}$. As a result,
$\dim\setS_{-k-1}\geq\dim\setS_{-k}$ for all $k$. Let us now
suppose \eqref{eqn:SET} fails. That means $\dim\setS_{-n}\leq
n-1$, which implies that there exists
$\ell\in\{0,\,1,\,\ldots,\,n-1\}$ such that
$\dim\setS_{-\ell-1}=\dim\setS_{-\ell}\leq n-1$. Since
$\setS_{-\ell-1}\supset\setS_{-\ell}$, both $\setS_{-\ell-1}$ and
$\setS_{-\ell}$ having the same dimension implies
$\setS_{-\ell-1}=\setS_{-\ell}$. By \eqref{eqn:dual} we can
therefore write
$\setS_{-\ell}^{\perp}=A^{T}\setS_{-\ell}^{\perp}\cap\setS_{0}^{\perp}$,
which implies $\setS_{-\ell}^{\perp}\subset
A^{T}\setS_{-\ell}^{\perp}$.  Since
$\dim\setS_{-\ell}^{\perp}\geq\dim A^{T}\setS_{-\ell}^{\perp}$ we
deduce that $\setS_{-\ell}^{\perp}=A^{T}\setS_{-\ell}^{\perp}$.
Recall that $\dim\setS_{-\ell}\leq n-1$. Therefore
$\dim\setS_{-\ell}^{\perp}\geq 1$. Then equality
$\setS_{-\ell}^{\perp}=A^{T}\setS_{-\ell}^{\perp}$ implies that
there exists an eigenvector $w\in\setS_{-\ell}^{\perp}$ and a
nonzero eigenvalue $\lambda\in\Complex$ such that $w^{T}A=\lambda
w^{T}$. Note also that $w^{T}B=0$ because
$\setS_{-\ell}^{\perp}\subset\setS_{0}^{\perp}=\N(B^{T})$. Hence
\eqref{eqn:PBH} fails.
\end{proof}

\begin{remark}\label{rem:dimension}
It is clear from the proof that if \eqref{eqn:SET} fails then
$\dim\setS_{-k}\leq n-1$ for all $k$.
\end{remark}

\section{Sets}\label{sec:sets}
In this section we define certain sets (more formally, {\em
equivalence classes}) associated with system~\eqref{eqn:system}.
For $x\in\X$ we define
\begin{eqnarray*}
[x]_{0}:=\mu(x,\,\U)\,.
\end{eqnarray*}
Note that for system~\eqref{eqn:sensible} we have
$\mu(x,\,u)=x+Bu$ and $[x]_{0}=x+\R(B)$. We then let for
$k=0,\,1,\,\ldots$
\begin{eqnarray*}
[x]_{-k-1}:=\mu([x]_{-k}^{-},\,\U)
\end{eqnarray*}
where
\begin{eqnarray*}
[x]_{-k}^{-}:=f^{-1}([f(x)]_{-k})\,.
\end{eqnarray*}

\begin{remark}\label{rem:subset}
Note that $[x]_{-k-1}\supset[x]_{-k}$ and
$[x]^{-}_{-k-1}\supset[x]^{-}_{-k}$ for all $x$ and $k$.
\end{remark}

The following two assumptions will be invoked in our main theorem.
In hope of making them appear somewhat meaningful and revealing
their generality we provide the conditions that they would boil
down to for linear systems.

\begin{assumption}\label{assume:X}
There exists $p\geq 1$ such that $[x]_{1-p}=\X$ for all $x\in\X$.
\end{assumption}

Assumption~\ref{assume:X} is equivalent to deadbeat
controllability for linear systems. Below result formalizes this.

\begin{theorem}
Linear system~\eqref{eqn:sensible} is deadbeat controllable if and
only if Assumption~\ref{assume:X} holds.
\end{theorem}

\begin{proof}
Let $\setS_{-k}$ for $k=0,\,1,\,\ldots$ be defined as in
Lemma~\ref{lem:subspace}. We claim the following.
\begin{eqnarray*}
[x]_{-k}=x+\setS_{-k}\implies[x]_{-k-1}=x+\setS_{-k-1}\,.
\end{eqnarray*}
To see that we write
\begin{eqnarray*}
[x]_{-k-1}
&=& A^{-1}[Ax]_{-k}+\setS_{0}\\
&=& A^{-1}(Ax+\setS_{-k})+\setS_{0}\\
&=& A^{-1}Ax+A^{-1}\setS_{-k}+\setS_{0}\\
&=& x+\N(A)+A^{-1}\setS_{-k}+\setS_{0}\\
&=& x+A^{-1}\setS_{-k}+\setS_{0}\\
&=& x+\setS_{-k-1}
\end{eqnarray*}
where we used the fact $A^{-1}\setS\supset\N(A)$ for any subspace
$\setS$. Hence our claim holds. Note that $[x]_{0}=x+\setS_{0}$.
Therefore, by induction, $[x]_{-k}=x+\setS_{-k}$ for all $k$.

Now suppose that the system is deadbeat controllable. Then, since
$[x]_{-k}=x+\setS_{-k}$, we see that Assumption~\ref{assume:X}
holds with $p=n+1$ thanks to Lemma~\ref{lem:subspace}. If however
the system is not deadbeat controllable, then by
Remark~\ref{rem:dimension} $\dim \setS_{-k}\leq n-1$ for all $k$.
Hence Assumption~\ref{assume:X} must fail.
\end{proof}

\begin{assumption}\label{assume:invariance}
$\xhat\in[x]_{0}$ implies $[\xhat]_{0}=[x]_{0}$ for all
$x,\,\xhat\in\X$.
\end{assumption}

\begin{theorem}
Assumption~\ref{assume:invariance} comes for free for linear
system~\eqref{eqn:sensible}.
\end{theorem}

\begin{proof}
Evident.
\end{proof}
\\

Last we let $[x]_{1}^{-}:=x$ and define map
$\pi:\X\times\X\to\{2-p,\,\ldots,\,-1,\,0,\,1\}$ as
\begin{eqnarray*}
\pi(\xhat,\,x) := \max\,\{2-p,\,\ldots,\,-1,\,0,\,1\}\quad
\mbox{subject to}\quad
[\xhat]_{0}\cap[x]^{-}_{\pi(\xhat,\,x)}\neq\emptyset
\end{eqnarray*}
where $p$ is as in Assumption~\ref{assume:X}.

\section{The result}\label{sec:main}

Below is our main theorem.

\begin{theorem}\label{thm:main}
Suppose Assumptions~\ref{assume:X}-\ref{assume:invariance} hold.
Then system
\begin{eqnarray*}
\hat{x}^{+}\in f([\xhat]_{0}\cap [x]^{-}_{\pi(\xhat,\,x)})
\end{eqnarray*}
is a deadbeat tracker.
\end{theorem}

\begin{proof}
We claim the following.
\begin{eqnarray}\label{eqn:claim}
\xhat\in[x]_{-\ell-1}\implies\xhat^{+}\in[f(x)]_{-\ell}
\end{eqnarray}
for all $\ell\in\{0,\,1,\,\ldots,\,p-2\}$. Let us prove our claim.
Note that $\xhat\in[x]_{-\ell-1}$ means
$\xhat\in\mu([x]_{-\ell}^{-},\,\U)$, which implies that there
exists
\begin{eqnarray}\label{eqn:eta1}
\eta\in[x]_{-\ell}^{-}
\end{eqnarray}
such that $\xhat\in\mu(\eta,\,\U)=[\eta]_{0}$. By
Assumption~\ref{assume:invariance} we have
$[\xhat]_{0}=[\eta]_{0}$. Then \eqref{eqn:obvious} yields
\begin{eqnarray}\label{eqn:eta2}
\eta\in[\xhat]_{0}\,.
\end{eqnarray}
From \eqref{eqn:eta1} and \eqref{eqn:eta2} we have
$[\xhat]_{0}\cap [x]^{-}_{-\ell}\neq\emptyset$. Therefore
$\pi(\xhat,\,x)\geq-\ell$. Employing Remark~\ref{rem:subset} we
can write
\begin{eqnarray*}
\hat{x}^{+}
&\in& f([\xhat]_{0}\cap [x]^{-}_{\pi(\xhat,\,x)})\\
&\subset& f([\xhat]_{0}\cap [x]^{-}_{-\ell})\\
&\subset& f([x]^{-}_{-\ell})\\
&=& f(f^{-1}([f(x)]_{-\ell}))\\
&\subset& [f(x)]_{-\ell}\,.
\end{eqnarray*}
Hence \eqref{eqn:claim} holds. By Assumption~\ref{assume:X} we
have $\xhat\in[x]_{1-p}$ for all $x,\,\xhat$. Therefore
\eqref{eqn:claim} and Remark~\ref{rem:subset} imply the existence
of $\ell^{*}\in\{0,\,1,\,\ldots,\,p-1\}$ such that
\begin{eqnarray}\label{eqn:tintin}
\psi(k,\,\xhat,\,x)\in[\phi(k,\,x)]_{0}
\end{eqnarray}
for all $k\geq\ell^{*}$. We can write
$\phi(k,\,x)\in[\psi(k,\,\xhat,\,x)]_{0}$ by \eqref{eqn:tintin}.
Consequently
\begin{eqnarray*}
\pi(\psi(k,\,\xhat,\,x),\,\phi(k,\,x))=1
\end{eqnarray*}
for all $k\geq\ell^{*}$. Then we deduce
$\psi(k,\,\xhat,\,x)=\phi(k,\,x)$ for all $k\geq p$.
\end{proof}

\begin{corollary}\label{cor:foralg}
Consider linear system~\eqref{eqn:sensible} with
$A\in\Real^{n\times n}$ and $B\in\Real^{n\times 1}$. Suppose
$(A,\,B)$ is a controllable\footnote{That is, $\mbox{rank}\ [B\
AB\ \ldots\ A^{n-1}B]=n$.} pair. Let $\setS_{-k}$ for
$k=0,\,1,\,\ldots$ be defined as in Lemma~\ref{lem:subspace}. Then
system
\begin{eqnarray*}
\xhat^{+}=A((\xhat+\setS_{0})\cap(x+A^{-1}\setS_{2-n}))
\end{eqnarray*}
is a deadbeat tracker.
\end{corollary}

\section{Examples}\label{sec:ex}

Here, for two third order nonlinear systems, we construct deadbeat
trackers. In the first example we study a simple homogeneous
system and show that the construction yields a homogeneous
feedback law. Hence our method may be thought to be somewhat {\em
natural} in the vague sense that the tracker it generates inherits
certain intrinsic properties of the system. In the second example
we aim to provide a demonstration on tracker construction for a
system that resides in a state space different than $\Real^{n}$.

\subsection{Homogeneous system}

Consider system~\eqref{eqn:system} with
\begin{eqnarray*}
f(x):=\left[\!\!
\begin{array}{c}
-x_{2}\\
x_{1}+x_{3}^{1/3}\\
x_{2}^{3}+x_{3}
\end{array}\!\!\right]\quad\mbox{and}\quad \mu(x,\,u):=\left[\!\!
\begin{array}{c}
x_{1}\\
x_{2}\\
x_{3}+u^3
\end{array}\!\!\right]
\end{eqnarray*}
where $x=[x_{1}\ x_{2}\ x_{3}]^{T}$. Let $\X=\Real^{3}$ and
$\U=\Real$. If we let dilation $\Delta_{\lambda}$ be
\begin{eqnarray*}
\Delta_{\lambda}:=\left[\!\!
\begin{array}{ccc}
\lambda&0&0\\
0&\lambda&0\\
0&0&\lambda^{3}
\end{array}\!\!\right]
\end{eqnarray*}
with $\lambda\in\Real$, then we realize that
\begin{eqnarray*}
f(\Delta_{\lambda}x)=\Delta_{\lambda}f(x)\quad\mbox{and}\quad
\mu(\Delta_{\lambda}x,\,\lambda u)=\Delta_{\lambda}\mu(x,\,u)\,.
\end{eqnarray*}
That is, the system is homogeneous \cite{rinehart09} with respect
to dilation $\Delta$. Before describing the relevant sets
$[x]_{-k}$ and $[x]_{-k}^{-}$ we want to mention that $f$ is
bijective and its inverse is
\begin{eqnarray}\label{eqn:finv}
f^{-1}(x)=\left[\!\!
\begin{array}{c}
x_{2}-(x_{1}^{3}+x_{3})^{1/3}\\
-x_{1}\\
x_{1}^{3}+x_{3}
\end{array}\!\!\right]
\end{eqnarray}
Now we are ready to construct our sets. By definition
$[x]_{0}=\mu(x,\,\U)$. Therefore we can write
\begin{eqnarray*}
[x]_{0}=\left\{ \left[\!\!\begin{array}{c} x_{1}\\
x_{2}\\
\alpha^{3}
\end{array}\!\!\right]:\alpha\in\Real
\right\}
\end{eqnarray*}
By \eqref{eqn:finv} we can then proceed as
\begin{eqnarray*}
[x]_{0}^{-}&=&f^{-1}([f(x)]_{0})\\
&=&f^{-1}\left(\left\{ \left[\!\!
\begin{array}{c}
-x_{2}\\
x_{1}+x_{3}^{1/3}\\
\alpha^{3}
\end{array}\!\!\right]:\alpha\in\Real
\right\}\right)\\
&=&\left\{f^{-1}\left( \left[\!\!
\begin{array}{c}
-x_{2}\\
x_{1}+x_{3}^{1/3}\\
\alpha^{3}
\end{array}\!\!\right]\right):\alpha\in\Real
\right\}\\
&=&\left\{\left[\!\!
\begin{array}{c}
x_{1}+x_{3}^{1/3}+(x_{2}^{3}-\alpha^{3})^{1/3}\\
x_{2}\\
-x_{2}^{3}+\alpha^{3}
\end{array}\!\!\right]:\alpha\in\Real
\right\}
\end{eqnarray*}
Since $[x]_{-1}=\mu([x]_{0}^{-},\,\U)$ we can write
\begin{eqnarray*}
[x]_{-1}&=&\left\{\left[\!\!
\begin{array}{c}
x_{1}+x_{3}^{1/3}+(x_{2}^{3}-\alpha^{3})^{1/3}\\
x_{2}\\
\beta^{3}
\end{array}\!\!\right]:\alpha,\,\beta\in\Real
\right\}\\
&=&\left\{\left[\!\!
\begin{array}{c}
\alpha\\
x_{2}\\
\beta^{3}
\end{array}\!\!\right]:\alpha,\,\beta\in\Real
\right\}
\end{eqnarray*}
We can now construct $[x]_{-1}^{-}$ as
\begin{eqnarray*}
[x]_{-1}^{-}&=&f^{-1}([f(x)]_{-1})\\
&=&\left\{ f^{-1}\left(\left[\!\!
\begin{array}{c}
\alpha\\
x_{1}+x_{3}^{1/3}\\
\beta^{3}
\end{array}\!\!\right]\right):\alpha,\,\beta\in\Real
\right\}\\
&=&\left\{\left[\!\!
\begin{array}{c}
x_{1}+x_{3}^{1/3}-(\alpha^{3}+\beta^{3})^{1/3}\\
-\alpha\\
\alpha^{3}+\beta^{3}
\end{array}\!\!\right]:\alpha,\,\beta\in\Real
\right\}\\
&=&\left\{\left[\!\!
\begin{array}{c}
x_{1}+x_{3}^{1/3}-\beta\\
\alpha\\
\beta^{3}
\end{array}\!\!\right]:\alpha,\,\beta\in\Real
\right\}
\end{eqnarray*}
Now it is easy to see that $\mu([x]_{-1}^{-},\,\Real)=\Real^{3}$.
Therefore $[x]_{-2}=\mu([x]_{-1}^{-},\,\U)=\X$ and
Assumption~\ref{assume:X} is satisfied with $p=3$. Observe also
that the following intersection
\begin{eqnarray*}
[\xhat]_{0}\cap[x]_{-1}^{-}&=&
\left\{ \left[\!\!\begin{array}{c} \xhat_{1}\\
\xhat_{2}\\
\gamma^{3}
\end{array}\!\!\right]:\gamma\in\Real
\right\}\cap\left\{\left[\!\!
\begin{array}{c}
x_{1}+x_{3}^{1/3}-\beta\\
\alpha\\
\beta^{3}
\end{array}\!\!\right]:\alpha,\,\beta\in\Real
\right\}\\
&=&\left[\!\!\begin{array}{c} \xhat_{1}\\
\xhat_{2}\\
(x_{1}-\xhat_{1}+x_{3}^{1/3})^{3}
\end{array}\!\!\right]
\end{eqnarray*}
is singleton, which means that $[\xhat]_{0}\cap
[x]^{-}_{\pi(\xhat,\,x)}=[\xhat]_{0}\cap [x]^{-}_{-1}$ for all
$\xhat$ and $x$. The dynamics of the deadbeat observer then read
\begin{eqnarray*}
\xhat^{+}&=&f([\xhat]_{0}\cap [x]^{-}_{-1})\\
&=&\left[\!\!\begin{array}{c} -\xhat_{2}\\
x_{1}+x_{3}^{1/3}\\
\xhat_{2}^{3}+(x_{1}-\xhat_{1}+x_{3}^{1/3})^{3}
\end{array}\!\!\right]
\end{eqnarray*}
In particular, solving for $u$ in $\mu(\xhat,\,u)=[\xhat]_{0}\cap
[x]^{-}_{-1}$ readily yields
$u=((x_{1}-\xhat_{1}+x_{3}^{1/3})^{3}-\xhat_{3})^{1/3}$. Therefore
\begin{eqnarray*}
\kappa(\xhat,\,x):=((x_{1}-\xhat_{1}+x_{3}^{1/3})^{3}-\xhat_{3})^{1/3}
\end{eqnarray*}
is a deadbeat feedback law. Notice that
\begin{eqnarray*}
\kappa(\Delta_{\lambda}\xhat,\,\Delta_{\lambda}x)=\lambda\kappa(\xhat,\,x)\,.
\end{eqnarray*}
That is, the feedback law and hence the tracker are homogeneous
with respect to dilation $\Delta$.

\subsection{Positive system}

Consider system~\eqref{eqn:system} with
\begin{eqnarray*}
f(x):=\left[\!\!
\begin{array}{c}
x_{1}x_{2}x_{3}\\
x_{3}/x_{1}\\
\sqrt{x_{1}x_{2}}
\end{array}\!\!\right]\quad\mbox{and}\quad \mu(x,\,u):=\left[\!\!
\begin{array}{c}
x_{1}/u\\
x_{2}u^{2}\\
x_{3}/u
\end{array}\!\!\right]
\end{eqnarray*}
Let $\X=\Real_{>0}^{3}$ and $\U=\Real_{>0}$. We construct the
relevant sets $[x]_{-k}$ and $[x]_{-k}^{-}$ as follows. We begin
with $[x]_{0}$.
\begin{eqnarray*}
[x]_{0}=\left\{ \left[\!\!\begin{array}{c} x_{1}/\alpha\\
x_{2}\alpha^2\\
x_{3}/\alpha
\end{array}\!\!\right]:\alpha>0
\right\}
\end{eqnarray*}
Note that
\begin{eqnarray*}
f^{-1}(x)=\left[\!\!\begin{array}{c} x_{1}/(x_{2}x_{3}^{2})\\
x_{2}x_{3}^{4}/x_{1}\\
x_{1}/x_{3}^{2}
\end{array}\!\!\right]
\end{eqnarray*}
Therefore
\begin{eqnarray*}
[x]_{0}^{-}&=&f^{-1}([f(x)]_{0})\\
&=&f^{-1}\left(\left\{ \left[\!\!
\begin{array}{c}
x_{1}x_{2}x_{3}/\alpha\\
x_{3}\alpha^{2}/x_{1}\\
\sqrt{x_{1}x_{2}}/\alpha
\end{array}\!\!\right]:\alpha>0
\right\}\right)\\
&=&\left\{f^{-1}\left(\left[\!\!
\begin{array}{c}
x_{1}x_{2}x_{3}/\alpha\\
x_{3}\alpha^{2}/x_{1}\\
\sqrt{x_{1}x_{2}}/\alpha
\end{array}\!\!\right]\right):\alpha>0
\right\}\\
&=&\left\{\left[\!\!
\begin{array}{c}
x_{1}/\alpha\\
x_{2}/\alpha\\
x_{3}\alpha
\end{array}\!\!\right]:\alpha>0
\right\}\\
\end{eqnarray*}
Since $[x]_{-1}=\mu([x]_{0}^{-},\,\U)$ we can write
\begin{eqnarray*}
[x]_{-1}=\left\{\left[\!\!
\begin{array}{c}
x_{1}/(\alpha\beta)\\
x_{2}\beta^{2}/\alpha\\
x_{3}\alpha/\beta
\end{array}\!\!\right]:\alpha,\,\beta>0
\right\}
\end{eqnarray*}
We can now construct $[x]_{-1}^{-}$ as
\begin{eqnarray*}
[x]_{-1}^{-}&=&f^{-1}([f(x)]_{-1})\\
&=&\left\{f^{-1}\left(\left[\!\!
\begin{array}{c}
x_{1}x_{2}x_{3}/(\alpha\beta)\\
x_{3}\beta^{2}/(x_{1}\alpha)\\
\alpha\sqrt{x_{1}x_{2}}/\beta
\end{array}\!\!\right]\right):\alpha,\,\beta>0
\right\}\\
&=&\left\{\left[\!\!
\begin{array}{c}
x_{1}/(\alpha^{2}\beta)\\
x_{2}\alpha^{4}/\beta\\
x_{3}\beta/\alpha^{3}
\end{array}\!\!\right]:\alpha,\,\beta>0
\right\}\\
\end{eqnarray*}
Now, it can be shown that
$\mu([x]_{-1}^{-},\,\Real_{>0})=\Real_{>0}^{3}$. Therefore
$[x]_{-2}=\mu([x]_{-1}^{-},\,\U)=\X$ and Assumption~\ref{assume:X}
is satisfied with $p=3$. Observe also that the following
intersection
\begin{eqnarray*}
[\xhat]_{0}\cap[x]_{-1}^{-}&=& \left\{ \left[\!\!\begin{array}{c} \xhat_{1}/\gamma\\
\xhat_{2}\gamma^2\\
\xhat_{3}/\gamma
\end{array}\!\!\right]:\gamma>0
\right\}\cap\left\{\left[\!\!
\begin{array}{c}
x_{1}/(\alpha^{2}\beta)\\
x_{2}\alpha^{4}/\beta\\
x_{3}\beta/\alpha^{3}
\end{array}\!\!\right]:\alpha,\,\beta>0
\right\}\\
&=&\left[\!\!\begin{array}{c} \xhat_{1}^{4/3}\xhat_{2}^{5/3}\xhat_{3}^{2}/(x_{1}^{1/3}x_{2}^{5/3}x_{3}^{2})\\
x_{1}^{2/3}x_{2}^{10/3}x_{3}^{4}/(\xhat_{1}^{2/3}\xhat_{2}^{7/3}\xhat_{3}^{4})\\
\xhat_{1}^{1/3}\xhat_{2}^{5/3}\xhat_{3}^{3}/(x_{1}^{1/3}x_{2}^{5/3}x_{3}^{2})
\end{array}\!\!\right]
\end{eqnarray*}
is singleton, which means that $[\xhat]_{0}\cap
[x]^{-}_{\pi(\xhat,\,x)}=[\xhat]_{0}\cap [x]^{-}_{-1}$ for all
$\xhat$ and $x$. The dynamics of the deadbeat observer then read
\begin{eqnarray*}
\xhat^{+}&=&f([\xhat]_{0}\cap [x]^{-}_{-1})\\
&=&\left[\!\!\begin{array}{c} \xhat_{1}\xhat_{2}\xhat_{3}\\
\xhat_{3}/\xhat_{1}\\
\xhat_{1}^{1/3}x_{1}^{1/6}x_{2}^{5/6}x_{3}/(\xhat_{2}^{1/3}\xhat_{3})
\end{array}\!\!\right]
\end{eqnarray*}
In particular,
$u=x_{1}^{1/3}x_{2}^{5/3}x_{3}^{2}/(\xhat_{1}^{1/3}\xhat_{2}^{5/3}\xhat_{3}^{2})$
solves $\mu(\xhat,\,u)=[\xhat]_{0}\cap [x]^{-}_{-1}$. Therefore
\begin{eqnarray*}
\kappa(\xhat,\,x):=x_{1}^{1/3}x_{2}^{5/3}x_{3}^{2}/(\xhat_{1}^{1/3}\xhat_{2}^{5/3}\xhat_{3}^{2})
\end{eqnarray*}
is a deadbeat feedback law.

\section{An algorithm for deadbeat gain}\label{sec:alg}

In this section we provide an algorithm to compute the deadbeat
feedback gain for a linear system~\eqref{eqn:standard} with scalar
input. (The algorithm follows from Corollary~\ref{cor:foralg}.)
Namely, given a controllable pair $(A,\,B)$ with
$A\in\Real^{n\times n}$ and $B\in\Real^{n\times 1}$, we provide a
procedure to compute the gain $K\in\Real^{1\times n}$ that renders
matrix $A-BK$ nilpotent. Below we let $\nal(\cdot)$ be some
function such that, given matrix $M\in\Real^{m\times n}$ whose
dimension of null space is $k$, $\nal(M)$ is some $n\times k$
matrix whose columns span the null space of $M$.

\begin{algorithm}\label{alg:db}
Given $A\in\Real^{n\times n}$ and $B\in\Real^{n\times 1}$, the
following algorithm generates deadbeat gain $K\in\Real^{1\times
n}$.
\begin{eqnarray*}
&& X = B\\
&& \mbox{{\bf for}}\quad i = 1:n-2\\
&& \qquad X = [A^{-1}X\ \  B]\\
&& \mbox{{\bf end}}\\
&& K_{2}=\frac{\nal((A^{-1}X)^{T})^{T}}{\nal((A^{-1}X)^{T})^{T}B}\\
&&K = K_{2}A
\end{eqnarray*}
\end{algorithm}

\begin{remark}
Matrix $K_{2}$ appearing in Algorithm~\ref{alg:db} is the deadbeat
gain for system~\eqref{eqn:sensible}. That is, matrix
$A(I-BK_{2})$ is nilpotent.
\end{remark}
Recall that for any subspace $\setS$ we can write
$(A^{-1}\setS)^{\perp}=A^{T}\setS^{\perp}$. Therefore for the case
when $A$ is not invertible or when matrix inversion is costly one
can use the below {\em dual} algorithm.

\begin{algorithm}\label{alg:dbdual}
Given $A\in\Real^{n\times n}$ and $B\in\Real^{n\times 1}$, the
following algorithm generates deadbeat gain $K\in\Real^{1\times
n}$.
\begin{eqnarray*}
&& X_{\rm perp} = \nal(B^{T})\\
&& \mbox{{\bf for}}\quad i = 1:n-2\\
&& \qquad X_{\rm perp}=\nal([\nal((A^{T}X_{\rm perp})^{T})\ \ B]^{T})\\
&& \mbox{{\bf end}}\\
&& K_{2}=\frac{(A^{T}X_{\rm perp})^{T}}{(A^{T}X_{\rm perp})^{T}B}\\
&&K = K_{2}A
\end{eqnarray*}
\end{algorithm}
For the interested reader we below give M{\small ATLAB} codes.
Algorithm~\ref{alg:db} can be realized through the following
lines.
\begin{verbatim}
X = B;
for i = 1:n-2
    X = [A^(-1)*X B];
end
Ktwo = null((A^(-1)*X)')'/(null((A^(-1)*X)')'*B);
K = Ktwo*A;
\end{verbatim}
Likewise, Algorithm~\ref{alg:dbdual} can be coded as follows.
\begin{verbatim}
Xperp = null(B');
for i = 1:n-2
    Xperp = null([null((A'*Xperp)') B]');
end
Ktwo = (A'*Xperp)'/((A'*Xperp)'*B);
K = Ktwo*A;
\end{verbatim}

\section{Conclusion}

For nonlinear systems a method to construct a deadbeat tracker is
proposed. The resultant tracker can be considered as a
generalization of the linear deadbeat tracker. The construction
makes use of sets that are generated iteratively. Through such
iterations, deadbeat feedback laws are derived for two academic
examples. Also, for computing the deadbeat gain for a linear
system with scalar input, an algorithm and its dual are given.

\bibliographystyle{plain}
\bibliography{references}

\begin{thebibliography}{10}

\bibitem{eldem94}
V.~Eldem and H.~Selbuz.
\newblock On the general solution of the state deadbeat control problem.
\newblock {\em IEEE Transactions on Automatic Control}, 39(5):1002--1006, 1994.

\bibitem{franklin82}
A.~Emami-Naeini and G.F. Franklin.
\newblock Deadbeat control and tracking of discrete-time systems.
\newblock {\em IEEE Transactions on Automatic Control}, 27(1):176--181, 1982.

\bibitem{grasselli80}
O.M. Grasselli, A.~Isidori, and F.~Nicolo.
\newblock Dead-beat control of discrete-time bilinear systems.
\newblock {\em International Journal of Control}, 32(1):31--39, 1980.

\bibitem{kucera84}
V.~Kucera and M.~Sebek.
\newblock On deadbeat controllers.
\newblock {\em IEEE Transactions on Automatic Control}, 29(8):719--722, 1984.

\bibitem{lax96}
P.D. Lax.
\newblock {\em Linear Algebra}.
\newblock Wiley-Interscience, 1996.

\bibitem{lewis82}
F.L. Lewis.
\newblock A general {R}iccati equation solution to the deadbeat control
  problem.
\newblock {\em IEEE Transactions on Automatic Control}, 27(1):186--188, 1982.

\bibitem{mullis72}
C.~Mullis.
\newblock Time optimal discrete regulator gains.
\newblock {\em IEEE Transactions on Automatic Control}, 17(2):265--266, 1972.

\bibitem{nesic99}
D.~Nesic.
\newblock Controllability for a class of simple {W}iener-{H}ammerstein systems.
\newblock {\em Systems \& Control Letters}, 36(1):51--59, 1999.

\bibitem{nesic98b}
D.~Nesic and I.M.Y. Mareels.
\newblock Dead beat controllability of polynomial systems: symbolic computation
  approaches.
\newblock {\em IEEE Transactions on Automatic Control}, 43(2):162--175, 1998.

\bibitem{nesic98a}
D.~Nesic, I.M.Y. Mareels, G.~Bastin, and R.~Mahony.
\newblock Output dead beat control for a class of planar polynomial systems.
\newblock {\em SIAM Journal on Control and Optimization}, 36(1):253--272, 1998.

\bibitem{oreilly81}
J.~O'Reilly.
\newblock The discrete linear time invariant time-optimal control
  problem---{A}n overview.
\newblock {\em Automatica}, 17(2):363--370, 1981.

\bibitem{rinehart09}
M.~Rinehart, M.~Dahleh, and I.~Kolmanovsky.
\newblock Value iteration for (switched) homogeneous systems.
\newblock {\em IEEE Transactions on Automatic Control}, 54(6):1290--1294, 2009.

\bibitem{sugimoto93}
K.~Sugimoto, A.~Inoue, and S.~Masuda.
\newblock A direct computation of state deadbeat feedback gains.
\newblock {\em IEEE Transactions on Automatic Control}, 38(8):1283--1284, 1993.

\bibitem{tuna11}
S.E. Tuna.
\newblock Deadbeat observer: construction via sets.
\newblock In {\em Proc. of the American Control Conference}, to appear, 2011.

\end{thebibliography}
\end{document}